\documentclass[10pt]{amsart}
\usepackage{amssymb,amsmath,amsfonts,amscd}

\theoremstyle{plain}
\newtheorem{theorem}{Theorem}[section]
\newtheorem{proposition}[theorem]{Proposition}
\newtheorem{definition}[theorem]{Definition}

\newtheorem{corollary}[theorem]{Corollary}
\newtheorem{remark}[theorem]{Remark}
\newtheorem{lemma}[theorem]{Lemma}

\def\eps{\epsilon}

\usepackage{pdfsync}

\newcommand{\rn}[1]{{\mathbb R}^{#1}}
\newcommand{\R}{\mathbb R}

\newcommand{\he}[1]{{\mathbb H}^{#1}}

\newcommand{\scal}[2]{\langle {#1} , {#2}\rangle}

\newcommand{\Scal}[2]{\langle {#1} \vert {#2}\rangle}
\newcommand{\scalp}[3]{\langle {#1} , {#2}\rangle_{#3}}

\newcommand{\ccheck}{{\vphantom i}^{\mathrm v}\!\,}
\newcommand{\mc}{\mathcal }

\newcommand{\N}{\mathbb N}

\usepackage[usenames]{color}

\begin{document}


\title[
] 
{$L^1$-Poincar\'e and Sobolev inequalities for differential forms 
  {in} Euclidean spaces
  \footnote{To appear in Science China Mathematics}
  }

\author[Annalisa Baldi, Bruno Franchi, Pierre Pansu]{
Annalisa Baldi\\
Bruno Franchi\\ Pierre Pansu
}



\maketitle

%

%

\section{Introduction}


The simplest form of Poincar\'e inequality in an open set $B\subset \rn n$ can be stated as follows: if $1\leq p<n$ 
there exists $C(B,p)>0$ such that for any
(say) smooth function $u$ on $\rn n$ there exists a constant $c_u$ such that 
$$
\|u-c_u\|_{L^q(B)} \leq C(n,p)\,\|\nabla u\|_{L^p(B)}   
$$
provided  $\frac{1}{p}-\frac{1}{q}=\frac{1}{n}$. Sobolev inequality is very similar, but in that case we
are dealing with compactly supported functions, so that the constant $c_u$   {can} be dropped. It is well
known  (Federer \& Fleming theorem, \cite{federer_fleming}) that for $p=1$ Sobolev inequality is equivalent to
the {classical} isoperimetric inequality (whereas Poincar\'e inequality correspond to   {classical} {\sl relative}
isoperimetric inequality).

Let us restrict for a while to the case $B=\rn n$, to investigate generalizations of these inequalities to differential forms.
It is easy to see that Sobolev and Poincar\'e inequalities are equivalent to the following problem:
we ask whether, given a closed differential $1$-form $\omega$ in $L^p(\R^n)$,  there exists   {a} $0$-form $\phi$ in $L^q(\R^n)$ with 
$\frac{1}{p}-\frac{1}{q}=\frac{1}{n}$ such   {that}
\begin{equation}\label{intro 1}
d\phi=\omega\qquad\mbox{and}\qquad \|\phi\|_q\leq C(n,p,h)\,\|\omega\|_p.
\end{equation}
Clearly, this problem can be formulated in general for $h$-forms $\omega$ in $L^p(\R^n)$ and we are lead to
look for $(h-1)$-forms $\phi$ in $L^q(\R^n)$ such that \eqref{intro 1} holds. This is the problem we have
in mind when we   {speak} about  Poincar\'e inequality for differential forms. 
When we   {speak} about Sobolev inequality, we have in mind compactly supported differential forms.

The case $p>1$ has been fully understood on bounded convex sets by Iwaniec \& Lutoborsky (\cite{IL}). On the other hand, 
in the   {full} space $\rn n$ an easy proof consists in putting $\phi=d^*\Delta^{-1}\omega$. Here, $\Delta^{-1}$ denotes the inverse of the Hodge Laplacian $\Delta=d^*d+dd^*$
and $d^*$ is the formal $L^2$-adjoint of $d$. The operator $d^*\Delta^{-1}$ is given by convolution with a homogeneous kernel of type $1$ in the terminology of \cite{folland_lectures} and \cite{folland_stein}, hence it is bounded from $L^p$ to $L^q$ if $p>1$. 
Unfortunately, this argument does not suffice for $p=1$ since, by \cite{folland_stein}, Theorem 6.10,  $d^*\Delta^{-1}$ maps $L^1$ only into the weak Marcinkiewicz space $L^{n/(n-1), \infty}$. Upgrading from $L^{n/(n-1), \infty}$ to $L^{n/(n-1)}$ is possible for functions
(see \cite{LN}, \cite{FGaW}, \cite{FLW_grenoble}), but the trick does not seem to generalize to differential forms.

Since   {the} case $p=1$ is the most relevant from a geometric point of view, we focus on that case. First of all,
we notice that Poincar\'e inequality with $p=1$ fails in   {top} degree unless a global integral inequality is satisfied.
Indeed for $h=n$ forms belonging to $L^1$ and with nonvanishing integral cannot be differentials of $
L^{n/(n-1)}$ forms, see \cite{Tripaldi}. In arbitrary degree, a   {similar} integral obstruction takes the form
 $\int\omega\wedge\beta=0$ for every constant coefficient form $\beta$ of complementary degree. 
 Therefore we introduce the subspace $L_0^1$ of $L^1$-differential forms satisfying these conditions.
However, in degree $n$   {assuming that} the integral constraint   {is satisfied} does   {not} suffice,   {as we shall see in Section \ref{parabolicity}}. On the other hand, for instance it follows from
\cite{BB2007} that Poincar\'e inequality holds in degree $n-1$.
We refer the reader to \cite{BFP3} for a discussion, in particular in connection with van Schaftingen's \cite{vS2014}) and Lanzani \& Stein's \cite{LS} results.

We can state our main results. 
We have:
\begin{theorem}[Global Poincar\'e and Sobolev inequalities]
 Let $h=1,\ldots,n-1$ and set $q=n/(n-1)$. For every closed $h$-form $\alpha\in L_0^1(\R^n)$, there exists an $(h-1)$-form $\phi\in L^{q}(\R^n)$, such that 
$$
d\phi=\alpha\qquad\mbox{and}\qquad \|\phi\|_{q}\leq C\,\|\alpha\|_{1}.
$$
Furthermore,  if $\alpha$ is compactly supported, so is $\phi$.

\end{theorem}

We also prove a local version of this inequality. 

\begin{corollary}
For $h=1,\ldots,n-1$, let $q=n/(n-1)$. Let $B\subset \rn n$ be a   {bounded open} convex set, and let $B'$ be an open
set, $B\Subset B'$. 
Then there exists $C=C(n,B,B')$ with the following property:
\begin{enumerate}
\item Interior Poincar\'e inequality. 
For every closed $h$-form $\alpha$ in $L^1(B')$, there exists an $(h-1)$-form $\phi\in L^q(B)$, such that 
$$
d\phi=\alpha_{|B}, \qquad\mbox{and}\qquad \|\phi\|_{L^q(B)}\leq C\,\|\alpha\|_{L^1(B')}.
$$
 
\item Sobolev inequality. For every closed $h$-form $\alpha\in L^1$ with support in $B$, there exists an $(h-1)$-form $\phi\in L^q$, with support in $B'$, such that 
$$
d\phi=\alpha \qquad\mbox{and}\qquad \|\phi\|_{L^q(B')}\leq C\,\|\alpha\|_{L^1(B)}.
$$
\end{enumerate}

We shall refer to the above inequality   {as} interior Poincar\'e and interior Sobolev inequality, respectively. The
world ``interior'' is meant to stress the loss of domain from $B'$ to $B$.

\end{corollary}

Remarkably, most of the techniques developed here can be adapted, in combination with other
{\sl ad hoc} arguments to deal with Poincar\'e and Sobolev inequalities in the Rumin complex
of Heisenberg groups (see \cite{BFP3}).

\section{Kernels}

\medskip

Throughout the present note our setting will be the Euclidean space
$\rn n$ with $n>2$.

\bigskip
  If $f$ is a real function defined in $\mathbb R^n$, we denote
    by $\ccheck f$ the function defined by $\ccheck f(p):=
    f(-p)$, and, if $T\in\mc D'(\mathbb R^n)$, then $\ccheck T$
    is the distribution defined by $\Scal{\ccheck T}{\phi}
    :=\Scal{T}{\ccheck\phi}$ for any test function $\phi$.

We remind also that the convolution $f\ast g$ is  well defined
when $f,g\in\mc D'(\rn n)$, provided at least one of them
has compact support. In this case the following identities
hold
\begin{equation}\label{convolutions var}
\Scal{f\ast g}{\phi} = \Scal{g}{\ccheck f\ast\phi}
\quad
\mbox{and}
\quad
\Scal{f\ast g}{\phi} = \Scal{f}{\phi\ast \ccheck g}
\end{equation}
 for any test function $\phi$.

Following \cite{folland_lectures}, Definition 5.3, we recall now the notion of {\it kernel of type $\mu$}
and some properties stated below in Proposition \ref{kernel}.

\begin{definition}\label{type} A kernel of type $\mu$ is a 
distribution $K\in \mc S'(\mathbb R^n)$, homogeneous  of degree $\mu-n$
that is smooth outside of the origin.

The convolution operator with a kernel of type $\mu$
$$
f\to f\ast K
$$
is still called an operator of type $\mu$.
\end{definition}

\begin{proposition}\label{kernel}
Let $K\in\mc S'(\rn r)$ be a kernel of type $\mu$ and let $D_j$ denote the $j$-th partial derivative in $\rn n$.
\begin{itemize}
\item[i)] $\ccheck K$ is again a kernel of type $\mu$;
\item[ii)] $D_jK$ and $KD_j $ are associated with  kernels of type $\mu-1$ for
$j=1,\dots,n$;
\item[iii)]  If $\mu>0$, then $K\in L^1_{\mathrm{loc}}(\mathbb R^n)$.
\end{itemize}
\end{proposition}

\begin{lemma}
 Let $g$ be a a kernel of type $\mu>0$,
and let $\psi\in \mc D(\rn n)$ be a test function.
Then $\psi \ast g $ is smooth on $\rn n$.
 
If, in addition, $R=R(D)$ is an homogeneous
 polynomial of degree $\ell\ge 0$ in $D:=(D_1,\dots,D_n)$,
 we have
 $$
R( \psi\ast g)(p)= O(|p|^{\mu-n-\ell})\quad\mbox{as }p\to\infty.
 $$
\end{lemma}

In particular, if $\psi\in \mc D(\rn n)$, and $K$ is a kernel of type
$\mu <n$, then both $\psi \ast K$
and all its derivatives belong to $L^\infty(\rn n)$

{  
\begin{corollary}\label{closed ex bis}
If  $K$ is a kernel of type $\mu\in (0,n)$,
$u\in L^1(\mathbb R^n)$ and
$\psi\in \mc D(\mathbb R^n)$, then
\begin{equation}
\Scal{u\ast K}{ \psi} = \Scal{ u}{\psi\ast\ccheck K}.
\end{equation}
\end{corollary}
In this equation, the left hand side is the action of a  distribution on a  test function, see formula \eqref{revised 1}, the right hand side is the inner product of an $L^1$ vector-valued function with an $L^\infty$ vector-valued function.
}

{  \begin{remark}
The conclusion of Corollary \ref{closed ex bis} still holds
if we assume $K\in L^1_{\mathrm{loc}}(\rn n)$, provided $u\in L^1(\mathbb R^n)$ is compactly
supported.
\end{remark}
}

\begin{lemma}\label{anuli}
Let $K$ be a kernel of type $\alpha\in (0,n)$, then for any $f\in L^1(\rn n)$ such that
$$
\int_{\rn n} f(y)\, dy = 0,
$$
we have:
$$
R^{-\alpha} \int_{B(0,2R)\setminus B(0,R)} |K\ast f| \, dx \longrightarrow 0\qquad\mbox{as $R\to\infty$.}
$$

\end{lemma}
\begin{proof}
If $R>1$  we have: 
\begin{equation*}\begin{split}
R^{-\alpha}  & \int_{R< |  x |  <2R} |K\ast f| \, dx
=
R^{-\alpha}\int_{R< |  x |  <2R} \,dx \Big| \int  \,  K(x-y) f(y) \, dy\Big|
\\&
= R^{-\alpha}\int_{R< |  x |  <2R} \,dx \Big| \int  \, \big[ K(x-y)-K(x)\big] f(y) \, dy\Big|
\\&
\le
R^{-\alpha}  \int  \,  | f(y)| \Big(  \int_{R< |  x |  <2R}  \Big| K(x-y)-K(x) \Big|  \,dx \Big)  \, dy
\\&
=R^{-\alpha} \int_{ |  y |  <\frac12 R} \,  | f(y)| \big( \cdots \big)  \, dy
+
R^{-\alpha} \int_{4R>  |  y |   >\frac12 R} \,  | f(y)| \big( \cdots \big)  \, dy
\\&
+
R^{-\alpha} \int_{ |  y |   > 4 R} \,  | f(y)| \big( \cdots \big)  \, dy
\\&
=: R^{-\alpha} I_{1}(R) + R^{-\alpha} I_{2}(R) + R^{-\alpha} I_{3}(R).
\end{split}\end{equation*}
Consider first the third term above. By homogeneity we have
$$
 I_{3}(R) \le C_K\,  \int_{ |  y |  >4 R} \,  | f(y)| \big( \int_{R< |  x |  <2R} ( |  x-y |  ^{-n+\alpha} +  |  x |  ^{-n+\alpha}  ) \,dx  \big)  \, dy
$$
Notice now that, if $ |  y |   >4R$ and $ R< |  x |  <2R$, then $ |  x-y |  \ge  |  y |   -  |  x |  
\ge 4 R -R   {\ge} \frac{3}2  |  x |  $. Therefore
\begin{equation*}\begin{split}
 |  x-y |  ^{-n+\alpha} & +  |  x |  ^{-n+\alpha} \le \left\{\big( \dfrac{2}{3}\big)^{n-\alpha} + 1\right\}  |  x |  ^{-n+\alpha},
\end{split}\end{equation*}
and then
$$
 \int_{R< |  x |  <2R} ( |  x-y |  ^{-n+\alpha} +  |  x |  ^{-n+\alpha}  ) \,dx \le C_\alpha \, R^\alpha.
$$
Thus
$$
R^{-\alpha} I_{3}(R) \le C_{K,\alpha}\,  \int_{ |  y |  >4R} \,  | f(y)|   \, dy \longrightarrow 0
$$
as $R\to\infty$.

Consider now the second term. Again we have
$$
 I_{2}(R) \le C_K\,  \int_{\frac12 R< |  y |  <4 R} \,  | f(y)| \big( \int_{R< |  x |  <2R} ( |  x-y |  ^{-n+\alpha} +  |  x |  ^{-n+\alpha}  ) \,dx  \big)  \, dy.
$$
Obviously, as above,
$$
\int_{R< |  x |  <2R}  |  x |  ^{-n+\alpha} \,dx \le C R^\alpha.
$$
Notice now that, if $\dfrac12 R <  |  y |   >4R$ and $ R< |  x |  <2R$, then $ |  x-y |  \le  |  x |   +  |  y |  
\ge 6R$. Hence
$$
 \int_{\frac12 R< |  y |  <4 R} \,  | f(y)| \big( \int_{ |  x-y |  <6R}  |  x-y |  ^{-n+\alpha} \,dx  \big)  \, dy \le CR^\alpha.
$$
Therefore 
$$
R^{-\alpha}I_{2}(R) \le C_K\,  \int_{\frac12 R< |  y |  <4 R} \,  | f(y)|\, dy \longrightarrow 0
$$
as $R\to\infty$.
Finally, if $ |  y |   < \frac{R}2$ and $ R< |  x |  <2R$ we have $ |  y |   < \frac12  |  x |  $ so that, by \cite{folland_stein}
Proposition 1.7 and Corollary 1.16,
\begin{equation*}\begin{split}
R^{-\alpha} I_{1}(R) & \le C_K\,  \int_{ |  y |  <\frac12 R} \,  | f(y)| \big( \int_{R< |  x |  <2R} \frac{ |  y |  }{ |  x |  ^{n-\alpha+1}} \,dx  \big)  \, dy
 \\&
 = C_K\,  \int_{\rn n} \,  | f(y)| |  y |   \chi_{[0,\frac12 R]}( |  y |  ) \big(R^{-\alpha} \int_{R< |  x |  <2R} \frac{1}{ |  x |  ^{n-\alpha+1}} \,dx  \big)  \, dy
  \\&
\le C_K\,  \int_{\rn n} \,  | f(y)| |  y |   \chi_{[0,\frac12 R]}( |  y |  )R^{-1} \, dy=: C_K\,  \int_{\rn n} \,  | f(y)|H_R( |  y |  ) \, dy.
\end{split}\end{equation*}
Obviously, for any fixed $y\in \he n $ we have $( |  y |  ) H_R ( |  y |  )\to 0$ as $R\to\infty$. On the other hand, 
$ | f(y)|H_R( |  y |  ) \le \frac12 |f(y)|$, so that, by dominated convergence theorem,
$$
R^{-\alpha} I_{1}(R) \longrightarrow 0
$$
as $R\to\infty$.

This completes the proof of the lemma.

\end{proof}

\begin{definition} Let $f$ be a measurable function on $\rn n$. If $t>0$ we set
$$
\lambda_f(t) = |\{|f|>t\}|.
$$
If $1\le p\le\infty$ and
$$
 \sup_{t>0} \lambda_f^p(t)  <\infty,
$$
we say that $f\in L^{p,\infty}(\rn n)$.
\end{definition}

\begin{definition}\label{M}
Following \cite{BBC}, Definition A.1, if $1<p<\infty$, we set
$$
\| u\|_{M^p} : = \inf \{C\ge 0 \, ; \, \int_K |u| \, dx \le C |K|^{1/p'}\;
\mbox{for all $L$-measurable set $K\subset {\R^{n}}$}\}.
$$
\end{definition}

By \cite{BBC}, Lemma A.2, we obtain

\begin{lemma}
If $1<p<\infty$, then
$$
\dfrac{(p-1)^p}{p^{p+1}}  \| u\|_{M^p }^p  \le \sup_{\lambda >0} \{\lambda^p | \{|u|>\lambda\} |\, \} \le  \| u\|_{M^p } ^p.
$$
In particular, if $1<p<\infty$, then  $M^p  = L^{p,\infty}(\rn n)$.
\end{lemma}

\begin{corollary}\label{marc alternative coroll}  If $1\le s <p$, then $M^p \subset L^s_{\mathrm{loc}} (\rn n)\subset L^1_{\mathrm{loc}} (\rn n)$.

\end{corollary}

\begin{proof} If $u\in M^p$ then $|u|^s\in M^{p/s}$, and we can conclude
thanks to Definition \ref{M}.

\end{proof}

\begin{lemma}\label{convolutions} Let $E$ be a kernel of type $\alpha\in (0,n)$. Then for all $f\in L^1(\rn n)$ we have $f\ast E\in M^{n/(n-\alpha)} $
and there exists $C>0$ such that 
$$
 \| f\ast E\|_{M^{n/(n-\alpha)}} \le C\|f \|_{L^1({\rn n})}
   $$
for all $f\in L^1(\rn n)$. In particular, by Corollary \ref{marc alternative coroll}, $f\ast E\in L^1_{\mathrm{loc}}$.
\end{lemma}

As in \cite{BFP2}, Lemma 4.4 and Remark 4.5, we have:

\begin{remark}
 
Suppose $0< \alpha<n$.
If $K$ is a kernel of type $\alpha$
and $\psi \in \mc D(\rn n)$, $\psi\equiv 1$ in a neighborhood of the origin, then
the statements of Lemma \ref{convolutions} still
hold if we replace $K$ by $(1-\psi )K$ or by $\psi K$. 

\end{remark}

\section{Differential forms and currents} Let $(dx_1,\dots, dx_{n}  )$ be the canonical basis of $(\rn n)^*$ and 
indicate as $\scalp{\cdot}{\cdot}{} $ the
inner product in $(\rn n)^*$  that makes $(dx_1,\dots, dx_{n}  )$ 
an orthonormal basis. We put
$       \bigwedge^0 (\rn n) := \R $
and, for $1\leq h \leq n$,
\begin{equation*}
\begin{split}
         \bigwedge^h(\rn n)& :=\mathrm {span}\{ dx_{i_1}\wedge\dots \wedge dx_{i_h}:
1\leq i_1< \dots< i_h\leq n\}
\end{split}
\end{equation*}
the linear space of the alternanting $h$-forms on $\rn n$.
If $I:= ({i_1},\dots , {i_h})$ with
$1\leq i_1< \dots< i_h\leq n$, we set $|I|:=h$ and
$$
dx^I:= dx_{i_1}\wedge\dots \wedge dx_{i_h}.
$$

We indicate as $\scalp{\cdot}{\cdot}{} $ also the
inner product in $\bigwedge^h(\rn n)$  that makes $(dx_1,\dots, dx_{n} )$ 
an orthonormal basis.

By translation,  $\bigwedge^h(\rn n)$ defines a fibre bundle over $\rn n$, still denoted by $\bigwedge^h(\rn n)$. A differential
form on $\rn n$ is a section of this fibre bundle.

Through this Note, if $0\le h\le n$ and $\mc U\subset \rn n$ is an open set, we denote by $\Omega^h(\mc U)$ the space of
differential $h$-forms on $\mc U$, and by $d:\Omega^h(\mc U)\to\Omega^{h+1}(\mc U)$
the exterior differential. Thus $(\Omega^\bullet(\mc U), d)$ is the de Rham complex
in $\mc U$ and any $u\in \Omega^h$ can be written as
$$
u= \sum_{|I|=h} u_I dx^I.
$$

\begin{definition}
 If $\mc U\subset\rn n$ is an open set and $0\le h\le n$,
we say that $T$ is a  $h$-current on $\mc U$
if $T$ is a continuous linear functional on $\mc D(\mc U,  \bigwedge^h(\rn n))$
endowed with the usual topology. We write $T\in \mc D'(\mc U,  \bigwedge^h(\rn n))$.
If
$ u \in L^1_{\mathrm{loc}}(\mc U,  \bigwedge^h(\rn n))$, then $u$ can be
identified canonically with a $h$-current $T_u$ through the
formula
\begin{equation*}
\Scal{T_u}{\varphi}:=\int_{\mc U} u\wedge \ast \varphi
=  \int_{\mc U} \scal{u}{\varphi}\, dx
\end{equation*} for any $\varphi\in \mc D(\mc U,  \bigwedge^h(\rn n))$.
\end{definition}

From now on, if there is no way to misunderstandings, and 
$ u \in L^1_{\mathrm{loc}}(\mc U,  \bigwedge^h(\rn n))$, we shall write $u$ instead of $T_u$.

Suppose now $u$ is sufficiently smooth (take for instance $u\in\mc \mc C^\infty(\rn n, \bigwedge^h(\rn n) )$. If $\phi\in\mc D(\rn n, \bigwedge^h(\rn n))$, then,
by the Green formula 
$$
\int_{\rn n} \scal{d  u}{\phi}\, dx =  \int_{\rn n} \scal{u}{d ^*\phi}\, dx.
$$ 
Thus, if $T\in \mc D'(\rn n, \bigwedge^h(\rn n)$. it is natural to set 
\begin{equation*}\begin{split}
\Scal{d T}{\phi} = \Scal{T}{d ^*\phi}
\end{split}\end{equation*}
for any $\phi\in \mc D(\rn n, \bigwedge^{h+1}(\rn n))$. 

Analogously, if $T\in \mc D'(\rn n, \bigwedge^h(\rn n))$, we set
 \begin{equation*}\begin{split}
\Scal{d ^*T}{\phi} = \Scal{T}{d \phi}
\end{split}\end{equation*}
for any $\phi\in \mc D(\rn n, \bigwedge^{h-1}(\rn n))$. 

Notice that, if $u \in L^1_{\mathrm{loc}}(\rn n, \bigwedge^{h}(\rn n)) $
$$
\Scal{u}{d ^*\phi} =  \int_{\rn n} u\wedge \ast d ^*\varphi
= (-1)^{h+1}  \int_{\rn n} u\wedge d ^* (\ast\varphi).
$$

A straightforward approximation   {argument} yields the following identity:
{  \begin{lemma}\label{june 6 eq:1}
Let   
$u\in L^1(\rn n, \bigwedge^{h+1}(\rn n))$ be a closed form, and let $K$ be a kernel of type $\mu\in (0,n)$. 
If $\psi\in \mc D(\rn n, \Omega^h)$, then 
\begin{equation}\label{closed eq:1}
\int \scal{u}{ d^*(\psi \ast K)}\, dx=0.
\end{equation}
\end{lemma}

}

 \begin{definition}
In $\rn n$,  we define
the Laplace-Beltrami operator $\Delta_{h}$  on $\Omega^h$   {by} 
\begin{equation*}
\Delta_{h}=
     dd^*+d^*d
     \end{equation*}
\end{definition}

Notice that $-\Delta_{0} = \sum_{j=1}^{2n}\partial_j^2$ is the usual Laplacian of
$\rn n$. 
\begin{proposition}[see e.g. \cite{jost} (2.1.28)]
If $u= \sum_{|I|=h} u_I dx^I$, then
$$
\Delta u = -  \sum_{|I|=h} (\Delta u_I) dx^I.
$$
\end{proposition}

 For sake of simplicity, since a basis  of $\bigwedge^h(\rn n)$
is fixed, the operator $\Delta_{h}$ can be identified with a diagonal matrix-valued map, still denoted
by $\Delta_{h}$,
\begin{equation}\label{matrix form}
\Delta_{h} = -(\delta_{ij}\Delta)_{i,j=1,\dots,\mathrm{dim}\, \bigwedge^h(\rn n)}: \mc D'(\rn{n}, \bigwedge^h(\rn n))\to D'(\rn{n}, \bigwedge^h(\rn n)),
\end{equation}
where $D'(\rn{n}, \bigwedge^h(\rn n))$ is the space of vector-valued distributions on $\rn n$.

If we denote by $\Delta^{-1}$ the matrix valued kernel
\begin{equation}\label{matrix form inverse}
\Delta_{h}^{-1} = -(\delta_{ij}\Delta^{-1})_{i,j=1,\dots,\mathrm{dim}\, \bigwedge^h(\rn n)}: \mc D'(\rn{n}, \bigwedge^h(\rn n))\to D'(\rn{n}, \bigwedge^h(\rn n)),
\end{equation}
then $\Delta_{h}^{-1} $ is a matrix-valued kernel of type 2 and
$$
\Delta_{h}^{-1} \Delta_{h}\alpha = \Delta_{h}\Delta_{h}^{-1} \alpha = \alpha \qquad\mbox{for all $\alpha\in\mc D(\rn n,\bigwedge^h(\rn n)$.}
$$

 We notice that, if $n>1$, since $ \Delta_h^{-1}$ is associated with a kernel of type 2 
 $ \Delta_h^{-1}f $ is well defined when  $f\in L^1(\he{n}, E_0^h)$. More precisely,  by 
Lemma \ref{convolutions} we have:

\begin{lemma}

 If $1\le h< n$, and $R=R(D)$ is   {a} homogeneous
 polynomial of degree $\ell=1$ in $D_1,\dots,D_n$,
 we have:
 $$
 \| f\ast R(D) \Delta_{h}^{-1}\|_{M^{n/(n-1)}} \le C\|f \|_{L^1(\rn n)}
   $$
   for all $f\in L^1(\rn n, \bigwedge^h(\rn n))$.

By Corollary \ref{marc alternative coroll}, in both cases $ f\ast R(D) \Delta_{h}^{-1}\in L^1_{\mathrm{loc}}(\rn n,\bigwedge^h(\rn n))$.
In particular, the map
\begin{equation}\label{L1-L1}
\Delta_{ h}^{-1}: L^1(\rn n, \bigwedge^h(\rn n))\longrightarrow L^1_{\mathrm{loc}}(\rn n, \bigwedge^h(\rn n))
\end{equation}
is continuous.
\end{lemma}

\begin{remark}\label{closed bis} By Lemma \ref{closed ex bis}, if
$u\in L^1(\rn n, \bigwedge^{h+1}(\rn n))$ and
$\psi\in \mc D(\rn n, \bigwedge^h(\rn n))$, then 
\begin{equation}\label{revised 1}
\Scal{\Delta^{-1}_{h} u}{ \psi} = \Scal{ u}{ \Delta^{-1}_{h}\psi}.
\end{equation}
\end{remark}
In this equation, the left hand side is the action of a matrix-valued distribution on a vector-valued test function, see formula \eqref{matrix form inverse}, 
whereasthe right hand side is the inner product of an $L^1$ vector-valued function with an $L^\infty$ vector-valued function.

A standard argument yields the following identities:

\begin{lemma}[see \cite{BFP2}, Lemma 4.11]\label{comm} 
If $\alpha\in\mc D(\rn n, \bigwedge^h(\rn n))$, then
\begin{itemize}
\item[i)]$
d \Delta^{-1}_{ h}\alpha = \Delta^{-1}_{ h+1} d\alpha$, \qquad $h=0,1,\dots, n-1$, 

\item[iv)]$d^*  \Delta^{-1}_{\mathbb H, h}\alpha = \Delta^{-1}_{\mathbb H, h-1} d^* \alpha$
 \qquad $ h=1,\dots, n$.

\end{itemize}
\end{lemma}

\begin{lemma}
 If $\alpha\in L^1(\rn n, \bigwedge^h(\rn n))$, then $\Delta^{-1}_{h}\alpha$
 is well defined and belongs to 
$ L^1_{\mathrm{loc}}(\rn n, \bigwedge^h(\rn n))$. If in addition $d \alpha=0$ in the distributional sense,
then the following result holds:
$$
d  \Delta^{-1}_{ h}\alpha = 0.
$$
 \end{lemma}

\begin{proof} Let $\phi\in \mc D(\rn n, \bigwedge^h(\rn n))$ be arbitrarily given.
By Lemma \ref{comm}, $d  \Delta^{-1}_{h}\phi =  \Delta^{-1}_{h}d \phi$.
Thus Remark \ref{closed bis} and Lemma \ref{june 6 eq:1} yield
$$
\Scal{d  \Delta^{-1}_{ h}\alpha}{\phi} = \Scal{\Delta^{-1}_{ h}\alpha}{d \phi}
= \Scal{\alpha}{\Delta^{-1}_{h}d \phi} =
\Scal{\alpha}{d  \Delta^{-1}_{ h}\phi} 
 = 0.
$$
\end{proof}

\section{$n$-parabolicity}
\label{parabolicity}

Recall that a noncompact Riemannian manifold $M$ is \emph{$p$-parabolic} if for every compact subset $K$ and every $\eps>0$, there exists a smooth compactly supported function $\chi$ on $M$ such that $\chi\geq 1$ on $K$ and 
$$
\int_{M}|d\chi|^p <\eps.
$$
It is well known that Euclidean $n$-space is $n$-parabolic (the relevant functions $\chi$ can be taken to be piecewise affine functions of $\log r$, where $r$ is the distance to the origin). It follows that Sobolev inequality in $L^n$ cannot hold, and, as we saw in the introduction, that the Poincar\'e inequality on $n$-forms fails as well.

Here, we explain an other consequence of $n$-parabolicity.

\begin{proposition}
\label{descends}
Let $\omega$ be a $k$-form in $L^1(\R^n)$. Assume that $\omega=d\phi$ where $\phi\in L^{n/(n-1)}(\R^n)$. Then, for every constant coefficient $n-k$-form $\beta$,
$$
\int_{\R^n}\omega\wedge\beta=0.
$$
\end{proposition}

\begin{proof}
Let $\chi_R$ be a smooth compactly supported function on $\R^n$ such that $\chi_R=1$ on $B(R)$ and $\int|d\chi_R|^n \leq \frac{1}{R}$. Let $\omega_R=d(\chi_R\phi)$. Then, since $\chi_R\phi\wedge\beta$ is compactly supported,
$$
\int_{\R^n}\omega_R\wedge\beta=\int_{\R^n}d(\chi_R\phi\wedge\beta)=0.
$$
Write $\omega_R=d\chi_R\wedge\phi+\chi_R\omega$. Since
$$
|\int_{\R^n}d\chi_R\wedge\phi\wedge\beta|
\leq \|d\chi_R\|_n\|\phi\|_{n/(n-1)}\|\beta\|_\infty
\leq\frac{C}{R^{1/n}}
$$
tends to $0$,
\begin{align*}
\int_{\R^n}\omega\wedge\beta
&=\lim_{R\to\infty}\int_{\R^n}\chi_R\omega\wedge\beta\\
&=-\lim_{R\to\infty}\int_{\R^n}\omega_R\wedge\beta=0.
\end{align*}

\end{proof}

In other words, the vanishing of all integrals $\int\omega\wedge\beta$ is a necessary condition for an $L^1$ $k$-form to be the differential of an $L^{n/(n-1)}$ $k-1$-form.

\section{Main results}

The following  estimate  provides primitives for globally defined closed $L^1$-forms, and
can be derived from Lanzani \& Stein inequality \cite{LS},   {approximating} closed forms in $L^1_0(\rn n),\bigwedge^h(\rn n))$
by means of closed compactly supported smooth form. The convergence of the approximation is guaranteed by Lemma \ref{anuli}.  
\begin{proposition}
Denote by $L^1_0(\rn n),\bigwedge^h(\rn n))$ the subspace of $L^1(\rn n,\bigwedge^h(\rn n))$
of forms with vanishing average, and by $\mc H^1(\rn n)$ the classic real Hardy space (see \cite{Stein}, Chapter 3). We have:
\begin{itemize}
\item[i)] if $h < n$, then
$$
\|d^*  \Delta_h^{-1} u\|_{L^{n/(n-1)}(\rn n)}\le C \|u\|_{L^{1}(\rn n)}\qquad\mbox{for all $u\in L^1_0(\rn n,\bigwedge^h(\rn n))\cap \ker d $};
$$
\item[ii)] if  $h= n$, then
$$
\|d^*  \Delta_{n}^{-1} u\|_{L^{n/(n-1)}(\rn n)}\le C \|u\|_{\mc H^{1}(\rn n)}\qquad\mbox{for all $u\in \mc H^1(\rn n)\cap \ker d $}
$$
\end{itemize}
We stress that the vanishing average assumption is necessary
 (see Proposition \ref{descends}). 

\end{proposition}

A standard approximation argument (akin to that of the   {classical} Meyer \& Serrin Theorem) yields
the following density result.

 \begin{lemma}
Let $B \subset \rn n$ an open set. If $0\le h\le n$, we set
$$
(L^1\cap d ^{-1}L^1)(B,\bigwedge^h(\rn n)):=\{\alpha\in L^1(B,\bigwedge^h(\rn n))\,;\,d \alpha\in L^1(B,\bigwedge^{h+1}(\rn n))\},
$$
endowed with the graph norm. Then $C^\infty(B,\bigwedge^h(\rn n))$ is dense in $(L^1\cap d ^{-1}L^1)(B,\bigwedge^h(\rn n))$.

\end{lemma}

Again through an approximation argument we can prove the following two lemmata:

  \begin{lemma}\label{L1 boundedness of Kd new}

If $K=d ^*\Delta_c^{-1}$, then:
\begin{itemize}
\item $K$ is a kernel of type $1$;

\item if $\chi$ is a smooth function with compact support in $B$, then the identity
$$
\chi=d  K\chi+Kd \chi
$$
holds on the space $(L^1\cap d ^{-1}(L^1)(B,\bigwedge^\bullet(\rn n) )$.
\end{itemize}

\end{lemma}

  \begin{lemma}
If $1\le h<n$, let
  $\psi\in  L^1(\rn n,\bigwedge^{h}(\rn n))$ be a compactly supported form with $d \psi \in L^1(\rn n,\bigwedge^{h+1}(\rn n))$,
  and let $\xi\in   {\bigwedge^{2n-h}}$ be a   {constant coefficient} form. Then
  $$
  \int_{\rn{n} }  d \psi\wedge \xi = 0.
  $$
\end{lemma}

We are able now to prove the following (approximate) homotopy formula for closed forms.

\begin{proposition}\label{smoothing}
Let $B\Subset B'$ be open sets in $\rn n$. For $h=1,\ldots,n-1$, take $q=n/(n-1)$. Then there exists a smoothing operator 
$S:L^1(B',\bigwedge^{h}(\rn n))\to W^{s,q}(B,\bigwedge^{h}(\rn n))$ 
for every $s\in\N$, and a bounded operator $T:L^1(B',\bigwedge^{h}(\rn n))\to L^q(B,\bigwedge^{h-1}(\rn n))$ such that, for closed $L^1$-forms $\alpha$ on $B'$, 
\begin{equation}\label{homotopy formula for closed forms}
\alpha=d T\alpha+S\alpha\qquad \mbox{on $B$}.
\end{equation}
In particular, $S\alpha$ is closed.

Furthermore, $T$ and $S$ merely enlarge by a small amount the support of compactly supported differential forms.
\end{proposition}

 \begin{proof}
{  Le us fix two open sets $B_0$ and $B_1$ with 
\begin{equation*}
B\Subset B_0 \Subset B_1\Subset B',
\end{equation*}
and a cut-off function $\chi \in \mc D(B_1)$,
 $\chi\equiv 1$ on  $B_0$. If $\alpha \in ( L^1 \cap d ^{-1})(B', \bigwedge^\bullet(\rn n))$, we set $\alpha_0=  \chi\alpha$, continued by zero outside $B_1$.
 Denote by $k$  the kernel associated with $K$  in Lemma \ref{L1 boundedness of Kd new}. 
 We consider a cut-off function $\psi_R$ supported in a $R$-neighborhood
of the origin, such that $\psi_R\equiv 1$ near the origin. Then we can write  
$k=k\psi_R + (1-\psi_R)k$. Thus, let us denote by $K_{R}$
 the convolution operator associated with $\psi_R k$.
By Lemma \ref{L1 boundedness of Kd new},
\begin{equation}\begin{split}\label{apr 20 eq:1}
\alpha_0 &=d  K\alpha_0+K_d  \alpha_0\\
&= d  K_{R} \alpha_0 + K_{R}d  \alpha_0 + S\alpha_0,
\end{split}\end{equation}
where $S_0$ is defined by
$$
S\alpha_0 :=  d ( (1-\psi_R)k\ast \alpha_0) + (1-\psi_R)k \ast d \alpha_0.
$$ 
 We set
$$
T_1\alpha := K_{R}\alpha_0, \qquad S_1\alpha:=  S\alpha_0.
$$

 If $\beta\in L^1(B_1,\bigwedge^h(\rn n))$, we set
$$
T_1 \beta:= K_{R}(\chi\beta)_{\big|_{B}}, \qquad S_1\alpha:=  {S\alpha_0}_{\big|_{B}}.
$$
We notice that, provided $R>0$ is small enough, the values of $T_1\beta$  do not depend on the continuation of $\beta$ outside $B_1$.
Moreover
$$
{K_{R} d \alpha_0}_{\big|_B} = K_{R} d  (\chi \alpha)_{\big|_B} = K_{R} (\chi d \alpha)_{\big|_B} =T_1(d \alpha),
$$
since $ d  (\chi \alpha)\equiv \chi d \alpha$ on $B_0$. 
Thus, by \eqref{apr 20 eq:1},
 $$
\alpha = d  T_1\alpha + T_1d  \alpha + S_1\alpha \qquad\mbox{in $B$}.
$$
Assume now that $d \alpha=0$. 
Then
 $$
\alpha = d  T_1\alpha +  S_1\alpha \qquad\mbox{in $B$}.
$$
Write $\phi=T_1\alpha\in L^1(B_0, \bigwedge^{h-1}(\rn n))$. By difference, $d \phi=\alpha-S_1\alpha\in L^1(B_0, \bigwedge^{h-1}(\rn n))$.

 }

The next step will consist of proving that $\phi\in L^q(B_0),  \bigwedge^{h-1}(\rn n)$,
``iterating'' the previous argument. Let us sketch how this iteration will work:
let $\zeta$ be a cut-off function supported in $B_0$, identically equal to $1$ in a neighborhood $\mc U$ of $B$, and set $\omega=d (\zeta\phi)$. 
Obviously, the form $\zeta\phi$ (and therefore also $\omega$) are defined on all $\rn n$ and are compactly
supported in $B_0$. In addition, $\omega$ is closed.
Suppose for a while we are able to prove that
\begin{itemize}
\item[a)]  $\omega \in L^1(\rn n, \bigwedge^{h}(\rn n))$;
\item[b)]  
$\| K_{0}\omega\|_{L^q(\rn n, \bigwedge^{h}(\rn n))} \le C\|\alpha\|_{L^1(B', \bigwedge^{h}(\rn n))}$,
\end{itemize}
and let us show how the argument can be carried out. 

First we stress that, if $R$ is small enough, then when $x\in B$, $K_{R}\omega(x)$
depends only on the restriction of $d \phi$ to $\mc U$, so that the map
$$
\alpha \to K_{R}\omega \big|_{B}
$$
is linear.

In addition, notice that $\omega=\chi\omega$, so that, by \eqref{apr 20 eq:1},
$$
d (\zeta\phi) = \omega = d K_{R}\omega + S\omega.
$$
Therefore in $B$
$$
\alpha-S_1\alpha = d \phi = d (\zeta\phi) = d K_{R}\omega + S_0\omega,
$$
and then in $B$
\begin{equation*}\begin{split}
\alpha &= d  (K_{R}\omega \big|_{B}) + S_1\alpha \big|_{B}+ S\omega \big|_{B}
\\&=: d (K_{R}(\chi\omega) \big|_{B}) + S\alpha = d T\alpha + S\alpha.
\end{split}\end{equation*}

First notice that the map $\alpha\to \omega=\omega(\alpha)$ is linear, and hence $T$ and
$S$ are linear maps. In addition, by b),
\begin{equation*}\begin{split}
\| T\alpha\|_{L^q(B,  \bigwedge^{h-1}(\rn n))} \le  \| K_{R}(\chi\omega)\|_{L^q(\rn n),  \bigwedge^{h}(\rn n)} =  \| K_{R}(\omega)\|_{L^q(\rn n,  \bigwedge^{h}(\rn n))}
\le C\, \|\alpha\|_{L^1(B',  \bigwedge^{h}(\rn n))}.
\end{split}\end{equation*}

As for the map $\alpha \to S\alpha$ we have just to point out that , when $x\in B$, $S\alpha(x)$
can be written as the convolution of $\alpha_0$ with a smooth kernel with bounded derivatives
of any order, and the proof is completed.

\end{proof}

Interior Poincar\'e and Sobolev inequalities follow now from the approximate homotopy formula for closed forms
\eqref{homotopy formula for closed forms}.

\begin{corollary}[Interior Poincar\'e and Sobolev inequalities]
Let $B\Subset B'$ open sets in $\rn n$, and assume $B$ is convex. For $h=1,\ldots,n-1$, let $q=n/(n-1)$. Then for every closed form 
$\alpha\in L^1(B', \bigwedge^{h}(\rn n))$, there exists an $(h-1)$-form $\phi\in L^q(B, \bigwedge^{h-1}(\rn n))$, such that 
$$
d \phi=\alpha_{|B}\qquad\mbox{and}\qquad \|\phi\|_{L^q(B,  \bigwedge^{h-1}(\rn n))}\leq C\,\|\alpha\|_{L^1(B',  \bigwedge^{h}(\rn n))}.
$$
Furthermore, if $\alpha$ is compactly supported, so is $\phi$.
\end{corollary}

\begin{proof}
By Proposition \ref{smoothing}, the $h$-form $S\alpha$ defined in \eqref{homotopy formula for closed forms} is closed and
belongs to  $L^{q}(B,\bigwedge^{h}(\rn n))$, with norm controlled by the $L^1$-norm of $\alpha$. Thus we can 
apply Iwaniec \& Lutoborski's  homotopy (\cite{IL}, Proposition 4.1) to obtain a differential $(h-1)$-form $\gamma$ on $B$
with norm in $W^{1,q}(B, \bigwedge^{h-1}(\rn n))$ controlled by the $L^q$-norm of $S\alpha$ and therefore from the
$L^1$-norm of $\alpha$. Set $\phi:= T\alpha + \gamma$. Clearly
$$
d\phi = dT\alpha + d\gamma =dT\alpha + S\alpha = \alpha.
$$
Then, by Proposition \ref{smoothing},
$$
\| \phi\|_{L^q(B, \bigwedge^{h-1}(\rn n))} \le C \big(\| \alpha \|_{L^1(B, \bigwedge^{h}(\rn n))} + \| S\alpha\|_{L^q(B, \bigwedge^{h}(\rn n))}\big)
\le C  \| \alpha \|_{L^1(B, \bigwedge^{h}(\rn n))}.
$$

\end{proof}

\section*{Acknowledgments}{A. B. and B. F. are supported by the University of Bologna, funds for selected research topics, and by MAnET Marie Curie
Initial Training Network, by 
GNAMPA of INdAM (Istituto Nazionale di Alta Matematica ``F. Severi''), Italy, and by PRIN of the MIUR, Italy.
\\
P.P. is supported by MAnET Marie Curie
Initial Training Network, by Agence Nationale de la Recherche, ANR-10-BLAN 116-01 GGAA and ANR-15-CE40-0018 SRGI. P.P. gratefully acknowledges the hospitality of Isaac Newton Institute, of EPSRC under grant EP/K032208/1, and of Simons Foundation.}

\bibliographystyle{amsplain}

\bigskip
\tiny{
\noindent
Annalisa Baldi and Bruno Franchi 
\par\noindent
Universit\`a di Bologna, Dipartimento
di Matematica\par\noindent Piazza di
Porta S.~Donato 5, 40126 Bologna, Italy.
\par\noindent
e-mail:
annalisa.baldi2@unibo.it, 
bruno.franchi@unibo.it.
}

\medskip

\tiny{
\noindent
Pierre Pansu 
\par\noindent Laboratoire de Math\'ematiques d'Orsay,
\par\noindent Universit\'e Paris-Sud, CNRS,
\par\noindent Universit\'e
Paris-Saclay, 91405 Orsay, France.
\par\noindent
e-mail: pierre.pansu@math.u-psud.fr
}

\end{document}